\documentclass[12pt,a4paper,reqno]{amsart}
\usepackage[utf8]{inputenc}
\usepackage{amsmath,amssymb,amsthm,amsfonts}
\usepackage{geometry}
\usepackage{hyperref}
\newtheorem{theorem}{Theorem}[section]
\newtheorem{lemma}[theorem]{Lemma}

\theoremstyle{definition}

\newtheorem{conjecture}[theorem]{Conjecture}
\theoremstyle{remark}

\makeatletter
\@namedef{subjclassname@2020}{%
  \textup{2020} Mathematics Subject Classification}
\makeatother
\newcommand{\F}{\mathbb{F}}

\title{\textbf{A Conjecture on Circular Permutations over Finite Fields}}
\author[X.-Q. Luo]{Xin-Qi Luo}
\address{School of Sciences, Changzhou Institute of Technology, Changzhou 213032, China}
\email{luoxq@czust.edu.cn}

\author[Y.-F. She]{Yue-Feng She}
\address{Department of Applied Mathematics, Nanjing Forestry University, Nanjing 210037, China}
\email{she.math@njfu.edu.cn}

\date{}
\thanks{The second author was supported by the Tian Yuan Mathematical Foundation (Grant No. 12526613).}

\begin{document}

\maketitle

\begin{abstract}
Let \(\mathbb{F}_q\) be a finite field with \(q>7\). Zhi-Wei Sun conjectured that for every \(a_0\in\mathbb{F}_q\), there is a circular permutation \((a_1,\dots,a_{q-1})\) of non-zero elements of \(\mathbb{F}_q\) such that \(a_0+a_i a_{i+1}\) is primitive for \(1\le i\le q-1\), where \(a_q=a_1\). In this paper, we confirm this conjecture for \(q>18\,888\,871\).  For \(a_0=0\), the conclusion holds for all \(q>4\). For \(a_0\ne0\), we apply the Chv\'atal--Erd\H{o}s theorem to obtain a Hamilton cycle.
\end{abstract}

\section{Introduction}

Let $\mathbb{F}_q$ be a finite field of order $q=p^f$, with $p$ prime and $f$ a positive integer. We denote the multiplicative group of $\mathbb{F}_q$ by $\mathbb{F}_q^*.$ The generators of $\mathbb{F}_q^*$ are called \emph{primitive elements}. Permutations satisfying arithmetic restrictions and primitive elements in finite fields lead to many interesting problems.
In 1982, Filz \cite{filz1982} conjectured that for every even $n$, there is a permutation $\tau$ of $\{1,\ldots,n\}$ such that
$$\tau(1)+\tau(2),\ \tau(2)+\tau(3),\ \dots,\ \tau(n)+\tau(1)$$
are all prime. This conjecture appears in Guy's book \cite{guy2004}. For recent progress on this problem, one may refer to \cite{chen2021, shang2025,she}.





Primitive elements and permutations lead to several questions.
In 2013, Sun \cite{sun2013} proposed a series of conjectures on permutations of finite fields.
For example, Sun and Hou \cite[Conjecture 3.5(i)]{sun2013} conjectured that there is a circular permutation of the elements of $\mathbb{F}_q$ for $q>7$ such that every adjacent sum is primitive.
This conjecture has recently been confirmed by the second author for $q>18\,888\,871$ \cite{she}.

In this paper, we study another conjecture of Sun \cite[Conjecture 3.8]{sun2013}.

\begin{conjecture} \label{conj:main}
Let $\mathbb{F}_q$ be a finite field with $q > 7$. For any given element $a_0 \in \mathbb{F}_q$, there exists a circular permutation $(a_1, \dots, a_{q-1})$ of all elements of $\mathbb{F}_q^*$ such that
$$
a_0 + a_1 a_2, \quad a_0 + a_2 a_3, \quad \dots, \quad a_0 + a_{q-2} a_{q-1}, \quad a_0 + a_{q-1} a_1
$$
are primitive elements of $\mathbb{F}_q$.
\end{conjecture}

For a graph $G$, its vertex set and edge set are denoted by $V(G)$ and $E(G)$, respectively. Given $a_0\in\mathbb{F}_q$, we define a graph $G_{q,a_0}$ with $V(G_{q,a_0})=\mathbb{F}_q^*$ by
$$
xy\in E(G_{q,a_0})\Longleftrightarrow x\neq y \text{ and } a_0+xy \text{ is primitive.}
$$
The desired circular permutation is exactly a Hamilton cycle in $G_{q,a_0}$.

\begin{theorem}
\label{thm:main}
For any prime power $q>18\,888\,871$ and any $a_0\in\mathbb{F}_q$, $G_{q,a_0}$ is Hamiltonian. When $a_0=0$, the conclusion holds for $q>4$.
\end{theorem}

\section{Preliminaries}
For a positive integer $n$, let $\omega(n)$, $\varphi(n)$ and $\mu(n)$ denote the number of distinct prime divisors of $n$, Euler's totient function and the M\"obius function, respectively. Define $$ W(n)=2^{\omega(n)}$$ to be the number of squarefree divisors of $n$.
We let $\widehat{\mathbb{F}_q^*}$ denote the group of multiplicative characters of $\mathbb{F}_q^*$. For any $\chi\in\widehat{\mathbb{F}_q^*}$, we extend $\chi$ to $\mathbb{F}_q$ by $\chi(0)=0$.
 The characteristic function of primitive elements in $\F_q$ is given by
\begin{equation}\label{eq:indicator}
\mathbf{1}_{\text{prim}}(x) = \frac{\varphi(q-1)}{q-1} \sum_{d \mid q-1} \frac{\mu(d)}{\varphi(d)} \sum_{\substack{\chi \in \widehat{\F}_q^* \\ \operatorname{ord}(\chi)=d}} \chi(x).
\end{equation}

 For any graph $G$, we write $\alpha(G)$ for its independence number, $\delta(G)$ for its minimum degree, and $\kappa(G)$ for its vertex connectivity. We will use the following theorem given by Chv\'atal and Erd\H{o}s \cite{chvatal1972}.

\begin{theorem}\label{lem:chvatal_erdos}
Let $G$ be a finite simple graph with at least three vertices. If $\kappa(G) \ge \alpha(G)$, then $G$ is Hamiltonian.
\end{theorem}

For $X, Y \subseteq \mathbb{F}_q^*$, define
\[
E(X, Y) = |\{(x, y) \in X \times Y : a_0 + xy \text{ is primitive}\}|.
\]

\begin{lemma}\label{lem:discrepancy}
For any subsets $X, Y \subseteq \F_q^*$, and any non-zero $a_0 \in \F_q^*$:
\begin{align*}\label{eq:discrepancy}
\left| E(X,Y) - \frac{\varphi(q-1)}{q-1} |X||Y| \right| \le \frac{\varphi(q-1)}{q-1}W(q-1)\sqrt{q|X||Y|}.
\end{align*}
\end{lemma}

\begin{proof}
By definition of $E(X,Y)$ and \eqref{eq:indicator},
\[
E(X,Y) =  \frac{\varphi(q-1)}{q-1} \sum_{d \mid q-1} \frac{\mu(d)}{\varphi(d)} \sum_{\substack{\chi \in \widehat{\F}_q^* \\ \operatorname{ord}(\chi)=d}} \sum_{x \in X} \sum_{y \in Y} \chi(a_0 + xy).
\]
The principal multiplicative character term contributes
\[
\frac{\varphi(q-1)}{q-1} |X||Y|-\frac{\varphi(q-1)}{q-1} T_0,
\]
where $T_0$ is the number of solutions to $a_0 + xy = 0$ in $X \times Y$. Since $a_0xy \neq 0$,
we have $T_0 \le \min(|X|, |Y|)\leq \sqrt{|X||Y|}$.

For any non-principal multiplicative character $\chi$, the Cauchy--Schwarz inequality gives
\[
\left| \sum_{x \in X} \sum_{y \in Y} \chi(a_0 + xy) \right|^2 \le |X| \sum_{x \in X} \left| \sum_{y \in Y} \chi(a_0 + xy) \right|^2 \le |X| \sum_{y_1, y_2 \in Y} \sum_{x \in \F_q} \chi(a_0 + xy_1) \overline{\chi}(a_0 + xy_2).
\]

For $y_1 = y_2$, the inner sum is $q - 1$.
For $y_1 \neq y_2$, the inner sum is
$$
\sum_{\substack{x\in\mathbb{F}_q\\x\neq -a_0y_2^{-1}}}\chi\left(\frac{xy_1+a_0}{xy_2+a_0}\right)=\sum_{\substack{x\in\mathbb{F}_q \\x\neq -a_0y_2^{-1}}}\chi\left(\frac{y_1}{y_2}+\frac{a_0(y_2-y_1)}{y_2(xy_2+a_0)}\right)=-\chi\left(\frac{y_1}{y_2}\right).
$$
Therefore
\begin{align*}
\left| \sum_{x \in X} \sum_{y \in Y} \chi(a_0 + xy) \right|^2\le& |X|\left((q-1)|Y|-\sum_{y_1\neq y_2\in Y}\chi(y_1)\overline{\chi}(y_2)\right)\\
=&|X|\left(q|Y|-\left|\sum_{y\in Y}\chi(y)\right|^2\right)\\
\leq&q|X||Y|.
\end{align*}

Since there are exactly $\varphi(d)$ multiplicative characters of order $d$ and $W(q-1)$ squarefree divisors of $q-1$, combining the above, we obtain
\begin{align*}
    \left| E(X,Y) - \frac{\varphi(q-1)}{q-1} |X||Y| \right|\le&\frac{\varphi(q-1)}{q-1}\sqrt{|X||Y|}+\frac{\varphi(q-1)}{q-1}(W(q-1)-1)\sqrt{q|X||Y|}\\
    \le& \frac{\varphi(q-1)}{q-1}W(q-1)\sqrt{q|X||Y|}.
\end{align*}
\end{proof}

\section{The case of $a_0\neq0$}
Throughout this section, we assume $a_0\neq0$. For simplicity, we write $G=G_{q,a_0}$.

\begin{lemma}\label{thm:alpha} The independence number of $G$ satisfies
\begin{equation*}\label{eq:alpha_bound}
\alpha(G) \le 2W(q-1)\sqrt{q}+1.
\end{equation*}
\end{lemma}

\begin{proof}
Let $I \subseteq\F_q^*$ be an independent set of maximum size $s = \alpha(G)$. We partition $I$ into two disjoint subsets $X$ and $Y$ with
\[
|X| = \left\lfloor \frac{s}{2} \right\rfloor \quad \text{and} \quad |Y| = \left\lceil \frac{s}{2} \right\rceil.
\]
Since $I$ is independent, there are no edges between $X$ and $Y$, so $E(X,Y) = 0$. Applying Lemma~\ref{lem:discrepancy}, we have
\[
 \frac{\varphi(q-1)}{q-1} |X||Y|  \le \frac{\varphi(q-1)}{q-1}W(q-1)\sqrt{q|X||Y|}.
\]
Hence
$$
\sqrt{\left\lfloor \frac{s}{2} \right\rfloor \left\lceil \frac{s}{2} \right\rceil}\le W(q-1)\sqrt{q}.
$$
If $s$ is even, we have
$$
s\leq 2W(q-1)\sqrt{q}.
$$
If $s$ is odd, we have
$$
s\leq\sqrt{s^2-1}+1\leq 2W(q-1)\sqrt{q}+1
$$
This completes the proof.
\end{proof}

\begin{lemma}\label{thm:kappa}
The vertex connectivity of $G$ satisfies
\begin{equation*}\label{eq:kappa_bound}
\kappa(G) \ge \varphi(q-1) -W(q-1)\sqrt{q}-1.
\end{equation*}
\end{lemma}

\begin{proof}
Let $S \subseteq V(G) = \F_q^*$ be a minimum vertex cut of $G$, that is, $|S| = \kappa(G)$. Given any vertex $v \in \F_q^*$, $a_0 + vy$ ranges over $\F_q\setminus\{a_0\}$
as $y$ ranges over $\F_q^*$. Since $a_0$ and $a_0+v^2$ may be primitive, we have
\begin{equation}\label{eq:delta_lower}
\delta(G) \ge \varphi(q-1) - 2.
\end{equation}

Let $A$ be a smallest connected component of $G - S$, and let $B = V(G) \setminus (A \cup S)$. For any $a\in A$, as there are no edges between $A$ and $B$, all its neighbors in $G$ lie in $A \cup S$ with at most $|A| - 1$ neighbors inside $A$. Therefore, $a$ must have at least $\deg(a) - (|A| - 1)$ neighbors in $S$ where $\deg(a)$ is the degree of $a$. Thus
\begin{equation}\label{eq:S_degree_bound}
\kappa(G) \ge \delta(G) - (|A| - 1) = \delta(G) - |A| + 1.
\end{equation}

In addition, applying Lemma~\ref{lem:discrepancy} to $X = A$ and $Y = B$ with $E(A, B) = 0$, we obtain that
\[
\frac{\varphi(q-1)}{q-1} |A||B| \le \frac{\varphi(q-1)}{q-1} W(q-1)\sqrt{q|A||B|}.
\]
Hence
$$
|A|\leq\sqrt{|A||B|}\leq W(q-1)\sqrt{q}.
$$
This, together with equations \eqref{eq:delta_lower} and \eqref{eq:S_degree_bound}, gives the desired lower bound for $\kappa(G)$.
\end{proof}

Inspired by the approach of Cohen, Oliveira e Silva, and Trudgian \cite{cohen2015blms}, we divide our proof into three cases.
\begin{proof}[Proof of Theorem~\ref{thm:main} when $a_0\neq0$]
By Theorem \ref{lem:chvatal_erdos}, Lemmas \ref{thm:alpha} and \ref{thm:kappa}, it suffices to show that
\begin{equation}\label{phi-lower-bound}
\varphi(q-1)\ge 3W(q-1)\sqrt q+2
\end{equation}
for any prime power $q> 18\,888\,871$.
Put $n=q-1$ and let $r=\omega(n)$. Write $$n = \prod_{j=1}^r \ell_j^{a_j},$$
where $\ell_1 < \dots < \ell_r$ are primes and $a_j \ge 1$.
We define $P_r=p_1p_2\ldots p_r$, where $p_j$ denotes the $j$-th prime.
Since
$$\frac{\varphi(n)}{n} = \prod_{j=1}^r\left(1 - \frac{1}{\ell_j}\right),$$
it increases as any prime divisor $\ell_j$ increases. Hence,
\[
\frac{\varphi(n)}{n}\ge\alpha_r:=\frac{\varphi(P_r)}{P_r}.
\]
Since
$$\frac{\varphi(n)}{\sqrt{n}} = \prod_{j=1}^r \ell_j^{(a_j-1)/2} \cdot \dfrac{\ell_j - 1}{\sqrt{\ell_j}},$$
 by the monotonicity of $(x-1)/\sqrt{x}$, we have
\[
\frac{\varphi(n)}{\sqrt{n}}\ge\beta_r:=\frac{\varphi(P_r)}{\sqrt{P_r}}.
\]



\noindent\textbf{Case 1: $r\ge9$.} Note that for every $k\ge10$,
\[
\frac{p_k-1}{\sqrt{p_k}}>2.
\]
Thus we obtain
\[
\frac{\beta_r}{2^r}
=
\frac{\beta_9}{2^r}\prod_{k=10}^r
\frac{p_k-1}{\sqrt{p_k}}
\ge
\beta_9\,2^{-9}>4.772.
\]
Since $n\ge 18\,888\,871$,
\begin{align*}
3\sqrt{1+\frac1n}
+\frac2{2^r\sqrt n}
<
3.001.
\end{align*}
Thus
$$
\frac{\varphi(n)}{\sqrt{n}}\ge \beta_r>3\cdot 2^r\sqrt{1+\frac1n}+\frac2{\sqrt n}
$$
and \eqref{phi-lower-bound} follows.

\medskip

\noindent\textbf{Case 2: \(r=8\).} Suppose that $n$ is squarefree. If \(\ell_8\ge43\), then
\[
n\ge
2\cdot3\cdot5\cdot7\cdot11\cdot13\cdot17\cdot43
=21\,951\,930
\]
and
\[
\frac{\varphi(n)}{\sqrt n}
\ge
\alpha_8\sqrt n\ge\frac{1658880}{9699690}\sqrt{21\,951\,930}>801.29.
\]
For fixed \(\ell_8<43\), we check all possible choices of 
\(\ell_1,\dots,\ell_7\) with
\[
n=\ell_1\cdots \ell_8>18\,888\,870.
\]
and \(n+1\) a prime power. Among all squarefree $n>18\,888\,870$, the minimum of $\varphi(n)/\sqrt{n}$ is at
\[
(\ell_1,\dots,\ell_8)=(2,3,5,7,11,13,23,29),
\]
for which $n=20\,030\,010,$ and $$\frac{\varphi(n)}{\sqrt{n}}>792.70.$$

Suppose  \(n\) is not squarefree. The least non-squarefree \(n>18\,888\,870\) with \(n+1\) a prime power is
\[
n=2^2\cdot3\cdot5\cdot7\cdot11\cdot13\cdot17\cdot31
=31\,651\,620.
\]
Therefore
\[
\frac{\varphi(n)}{\sqrt n}
\ge
\alpha_8\sqrt n
>962.01.
\]
In summary, we have
\[
\frac{\varphi(n)}{\sqrt n}
>792.70.
\]
Moreover,
$$
3\cdot 2^8\sqrt{1+\frac1n}+\frac2{\sqrt n}<768.01
$$
for $n>18\,888\,870$. Thus \eqref{phi-lower-bound} holds.

\medskip
\noindent\textbf{Case 3: \(1\le r\le7\).} Since
$$\alpha_r\ge\alpha_7=\frac{3072}{17017},$$
and $n>18\,888\,870$, we have
\[
\frac{\varphi(n)}{\sqrt n}
\ge \alpha_7\sqrt n
>784.56>384.01>3\cdot 2^r\sqrt{1+\frac1n}+\frac2{\sqrt n}.
\]
So \eqref{phi-lower-bound} holds.

\medskip
Combining Cases 1, 2 and 3, $G$ is Hamiltonian when $q> 18\,888\,871$.
\end{proof}

\section{The case of $a_0=0$}

Our construction for odd prime powers $q$ is inspired by Sinha, Garg and Singh \cite{sinha2011}.

\begin{proof}[Proof of Theorem~\ref{thm:main} when $a_0=0$]
Let $g$ be a primitive element of $\mathbb{F}_q$ and let $n=q-1$. Any $a_i \in \mathbb{F}_q^*$ can be represented uniquely by its exponent with respect to $g$, that is,
\begin{equation*}
a_i = g^{x_i} \quad \text{for some unique } x_i \in \{0,1,\ldots,n-1\}.
\end{equation*}
The adjacent product $a_i a_{i+1} = g^{x_i + x_{i+1}}$ is a primitive element of $\mathbb{F}_q$ if and only if
\begin{equation}
\gcd(x_i + x_{i+1}, \, n) = 1.\nonumber
\end{equation}
We construct the Hamilton cycle by considering the parity of $n$ separately.

If $n$ is even, consider the sequence
\begin{equation*}
(x_1,x_2,\ldots,x_n,x_1)=(0, \, n-1, \, 2, \, n-3, \, 4, \, n-5, \, \dots, \, n-2, \, 1, \, 0).
\end{equation*}
It yields a Hamilton cycle in $G_{q,0}$.

If $n$ is odd, write $n=2h+1$, where $h \geq 2$. When $h=2$, we have $q=6$, which is not a prime power.
For $h=3$, we have
$$
(x_1,x_2,\ldots,x_7,x_1)=(6,0,5,1,3,2,4,6).
$$
and every adjacent sum is coprime to $7$. For $h \geq 4,$ consider the sequence
$$
(x_1,x_2,\ldots,x_n,x_1)=(n-1,0,n-2,1,h,h-1,h+1,h-2,h+2,\ldots,2,n-3,n-1).
$$
In this sequence, every adjacent sum is congruent to $n-1$, $n-2$, $n-4$ or $h+1$ modulo $n$. Since
\begin{equation*}
\gcd(n-1, n) = \gcd(n-2, n) =\gcd(n-4, n) = \gcd(h+1,n) = 1,
\end{equation*}
this leads to a Hamilton cycle in $G_{q,0}$. This completes the proof.
\end{proof}

\end{document}